\documentclass[reqno,centertags, 12pt]{amsart}
\usepackage{amsmath,amsthm,amscd,amssymb}
\usepackage{latexsym}
\usepackage{graphicx}
\newcommand{\bbC}{{\mathbb{C}}}
\newcommand{\bbN}{{\mathbb{N}}}

\newcommand{\lb}{\label}

\newcommand{\bi}{\bibitem}

\newcommand{\beq}{\begin{equation}}
\newcommand{\eeq}{\end{equation}}
\newcommand{\ba}{\begin{align}}
\newcommand{\ea}{\end{align}}

\newcounter{smalllist}

\allowdisplaybreaks
\numberwithin{equation}{section}

\newtheorem{theorem}{Theorem}[section]

\newtheorem*{p2.1}{Proposition 2.1}
\newtheorem{proposition}[theorem]{Proposition}

\newtheorem{corollary}[theorem]{Corollary}
\theoremstyle{definition}

\newtheorem*{tA}{Theorem A}
\newtheorem*{tB}{Theorem B}

\theoremstyle{remark}
\newtheorem*{remark}{Remark}

\theoremstyle{definition}

\newtheorem*{acknowledgement}{Acknowledgment}

\begin{document}
\title[A GENERALIZATION OF MACMAHON'S FORMULA]{A GENERALIZATION OF MACMAHON'S FORMULA}
\author[M.~Vuleti\'{c}]{Mirjana Vuleti\'{c}$^{*}$}

\thanks{$^*$ Mathematics 253-37, California Institute of Technology, Pasadena, CA 91125.
E-mail: vuletic@caltech.edu }

\begin{abstract}
We generalize the generating formula for plane partitions known as
MacMahon's formula as well as its analog for strict plane
partitions.
We give a 2-parameter generalization of these formulas related to Macdonald's symmetric functions. The formula is especially simple in the Hall-Littlewood case. We also give a bijective proof of the analog of MacMahon's formula for strict plane partitions.
\end{abstract}

\maketitle

\section{Introduction} \lb{s1}

A plane partition is a Young diagram filled with positive
integers that form nonincreasing rows and columns. Each plane
partition can be represented as a finite two sided sequence of
ordinary partitions
$(\dots,\lambda^{-1},\lambda^0,\lambda^1,\dots)$, where $\lambda^0$
corresponds to the ordinary partition on the main diagonal and
$\lambda^k$ corresponds to the diagonal shifted by $k$. A plane
partition whose all diagonal partitions are strict ordinary
partitions (i.e. partitions with all distinct parts) is called a
{\it strict} plane partition.  Figure \ref{PlanePartition} shows two standard
ways of representing a plane partition. Diagonal partitions are
marked on the figure on the left.

\begin{figure}[htp!]
\centering \includegraphics[height=5.8cm]{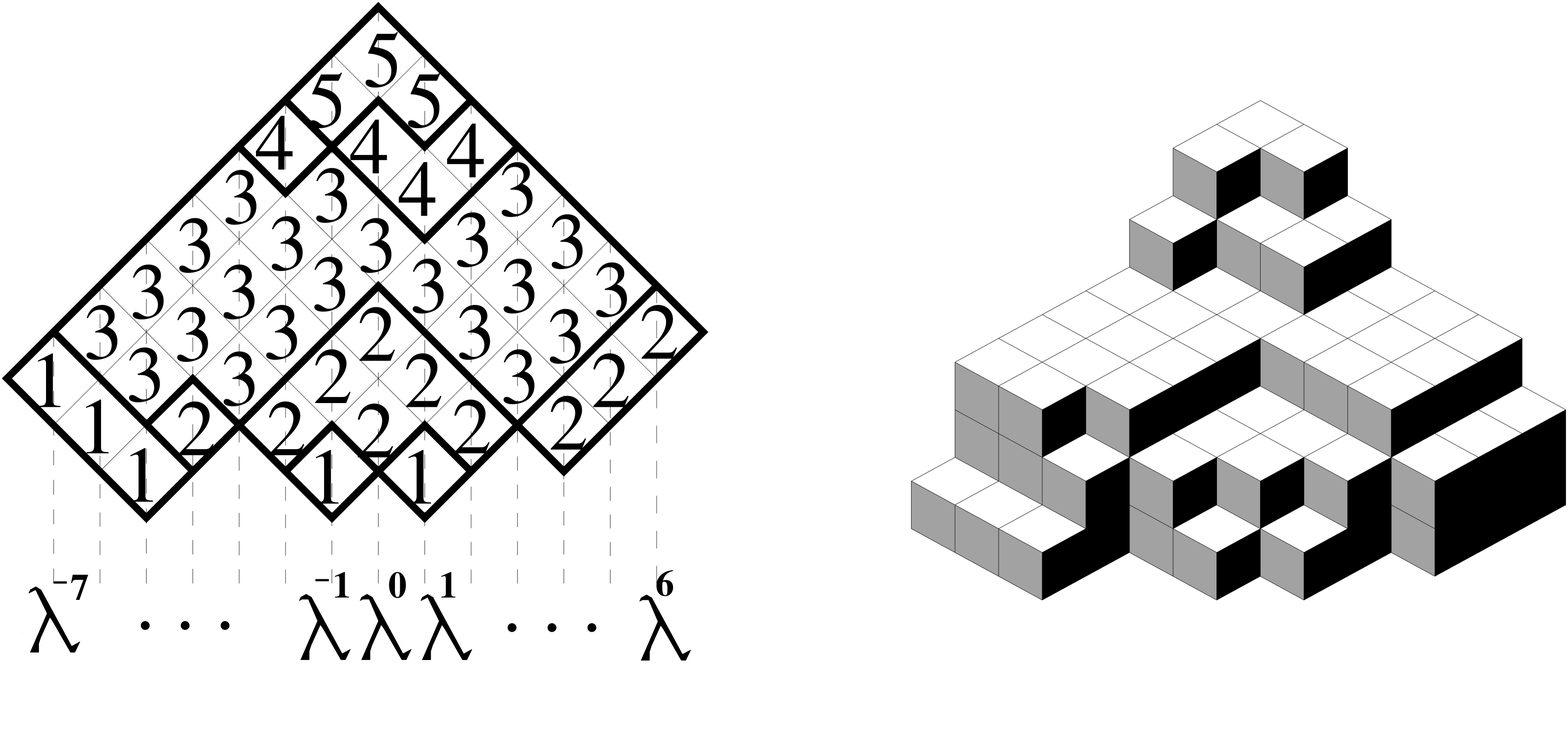}
\caption{A plane partition}
\label{PlanePartition} \end{figure}

For a plane partition $\pi$ one defines the weight $|\pi|$ to be the
sum of all entries. A {\it connected component} of a plane partition
is the set of all connected boxes of its Young diagram that are
filled with a same number. We denote the number of connected
components of $\pi$ with $k(\pi)$. For the example from Figure
\ref{PlanePartition}  we have $k(\pi)=10$ and its connected
components are shown in Figure \ref{PlanePartition} (left- bold
lines represent boundaries of these components, right- white
terraces are connected components).

Denote the set of all plane partitions with $\mathcal{P}$ and with
$\mathcal{P}(r,c)$ we denote those that have zero $(i,j)$th entry
for $i>r$ and $j>c$. Denote the set of all strict plane partitions
with $\mathcal{SP}$.

A generating function for plane partitions is given by the famous
MacMahon's formula (see e.g. 7.20.3 of \cite{S}):
\begin{equation}\lb{MacMahonac}
\sum_{\pi \in \mathcal{P}}s^{|\pi|}= \prod_{n=1}^{\infty}
\left(\frac{1}{1-s^{n}}\right)^n.
\end{equation}

Recently, a generating formula for the set of strict plane
partitions was found in \cite{FW} and \cite{V}:
\begin{equation}\lb{ShiftedMacMahonac}
\sum_{\pi \in \mathcal{SP}}2^{k(\pi)}s^{|\pi|}= \prod_{n=1}^{\infty}
\left(\frac{1+s^n}{1-s^{n}}\right)^n.
\end{equation}
We refer to it as  the shifted MacMahon's formula.


In this paper we generalize both formulas (\ref{MacMahonac}) and
(\ref{ShiftedMacMahonac}). Namely, we define a polynomial $A_\pi(t)$
that gives a generating formula for plane partitions of the form
$$
\sum_{\pi \in \mathcal{P}(r,c)}A_\pi(t)s^{|\pi|}=
\prod_{i=1}^{r}\prod_{j=1}^{c} \frac{1-ts^{i+j-1}}{1-s^{i+j-1}}
$$
with the property that $A_\pi(0)=1$ and
$$
A_\pi(-1)=
\begin{cases}
2^{k(\pi)},&\pi \text{ is a strict plane partition}, \\
0,& \text{otherwise}.
\end{cases}
$$
We further generalize this and find a rational function
$F_{\pi}(q,t)$ that satisfies
$$
\sum_{\pi \in
\mathcal{P}(r,c)}F_\pi(q,t)s^{|\pi|}=\prod_{i=1}^{r}\prod_{j=1}^c
\frac{(ts^{i+j-1};q)_\infty}{(s^{i+j-1};q)_{\infty}},
$$
where
$$
(s;q)_\infty=\prod_{n=0}^{\infty}(1-sq^n)
$$
and $F_\pi(0,t)=A_\pi(t)$. We describe $A_\pi(t)$ and $F_\pi(q,t)$
below.

In order to describe $A_\pi(t)$ we need more notation. If a box
$(i,j)$ belongs to a connected component $C$ then we define its {\it
level} $h(i,j)$  as the smallest positive integer such that
$(i+h,j+h)$ does not belong to $C$. A {\it border component} is
a connected subset of a connected component where all boxes have the
same level. We also say that this border component is of this level.
For the example above, border components and their levels are shown
in Figure \ref{BorderComponents}.

\begin{figure} [htp!]
\centering \includegraphics[height=6cm]{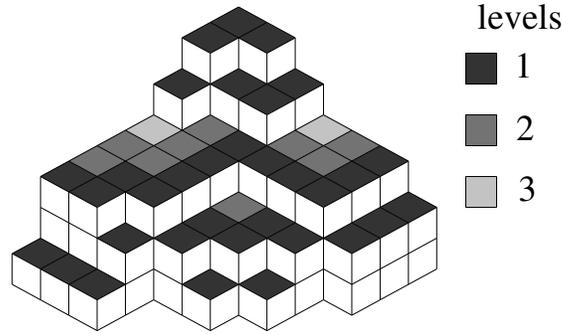}
\caption{Border Components}
\label{BorderComponents} \end{figure}

For each connected component $C$ we define a sequence
$(n_1,n_2,\dots)$ where $n_i$ is the number of $i$-level border
components of $C$. We set
$$
P_C(t)=\prod_{i\geq1}(1-t^{i})^{n_i}.
$$
Let $C_1,C_2,\dots,C_{k(\pi)}$ be connected components of $\pi$. We
define
$$
A_\pi(t)=\prod_{i=1}^{k(\pi)}P_{C_i}(t).
$$
For the example above $A_\pi(t)=(1-t)^{10}(1-t^2)^3(1-t^3)^2$.

$F_{\pi}(q,t)$ is defined as follows. For nonnegative integers $n$
and $m$ let
$$
f(n,m)= \begin{cases} \displaystyle
\prod_{i=0}^{n-1}\frac{1-q^{i}t^{m+1}}{1-q^{i+1}t^{m}},&\;\;\;n\geq1,\\
\;1,&\;\;\;n=0.
\end{cases}
$$
Here $q$ and $t$ are parameters.

Let $\pi\in \mathcal{P}$ and let $(i,j)$ be a box in its support
(where the entries are nonzero). Let $\lambda$, $\mu$ and $\nu$ be
ordinary partitions defined by
\begin{equation}\lb{lmn}
\begin{tabular}{l}
$\lambda=(\pi(i,j),\pi(i+1,j+1),\dots)$,\\
$\mu=(\pi(i+1,j),\pi(i+2,j+1),\dots)$,\\
$\nu=(\pi(i,j+1),\pi(i+1,j+2),\dots).$\\
\end{tabular}
\end{equation}
To the box $(i,j)$ of $\pi$ we associate
$$
F_\pi(i,j)(q,t)=\prod_{m=0}^{\infty}\frac{f(\lambda_1-\mu_{m+1},m)f(\lambda_1-\nu_{m+1},m)}
{f(\lambda_1-\lambda_{m+1},m)f(\lambda_1-\lambda_{m+2},m)}.
$$
Only finitely many terms in this product are different than 1.

To a plane partition $\pi$ we associate a function $F_\pi(q,t)$
defined by
\begin{equation}
F_\pi(q,t)=\prod_{(i,j)\in \pi}F_\pi(i,j)(q,t).
\end{equation}
For the example above
$$
F_\pi(0,0)(q,t)=\frac{1-q}{1-t}\cdot\frac{1-q^3t^2}{1-q^2t^3}\cdot\frac{1-q^5t^4}{1-q^4t^5}\cdot\frac{1-q^3t^5}{1-q^4t^4}.
$$

Two main results of our paper are
\begin{tA}
(Generalized MacMahon's formula; Macdonald's case)
$$
\sum_{\pi \in
\mathcal{P}(r,c)}F_\pi(q,t)s^{|\pi|}=\prod_{i=1}^{r}\prod_{j=1}^c
\frac{(ts^{i+j-1};q)_\infty}{(s^{i+j-1};q)_{\infty}},
$$
In particular,
$$
\sum_{\pi \in \mathcal{P}}F_\pi(q,t)s^{|\pi|}=\prod_{n=1}^{\infty}
\left[\frac{(ts^{n};q)_\infty}{(s^{n};q)_{\infty}}\right]^n.
$$
\end{tA}
\begin{tB}(Generalized MacMahon's formula; Hall-Littlewood's case)
\begin{equation*}
\sum_{\pi \in \mathcal{P}(r,c)}A_\pi(t)s^{|\pi|}=
\prod_{i=1}^{r}\prod_{j=1}^{c} \frac{1-ts^{i+j-1}}{1-s^{i+j-1}}.
\end{equation*}
In particular,
\begin{equation*}
\sum_{\pi \in \mathcal{P}}A_\pi(t)s^{|\pi|}= \prod_{n=1}^{\infty}
\left(\frac{1-ts^n}{1-s^{n}}\right)^n.
\end{equation*}
\end{tB}

Clearly, the second formulas (with summation over $\mathcal{P}$) are
limiting cases of the first ones as $r,c \to \infty$.

%
%

The proof of Theorem A was inspired by \cite{OR} and \cite{V}. It
uses a special class of symmetric functions called skew Macdonald
functions. For each $\pi \in \mathcal{P}$ we introduce a weight
function depending on several specializations  of the algebra of
symmetric functions. For a suitable choice of these
specializations the weight functions become $F_\pi(q,t)$.

We first prove Theorem A and Theorem B is obtained as a corollary of
Theorem A after we show that $F_\pi(0,t)=A_\pi(t)$.

Proofs of formula (\ref{ShiftedMacMahonac}) appeared in \cite{FW}
and \cite{V}. Both these proofs rely on skew Schur functions  and a
Fock space corresponding to strict plane partitions. In this paper
we also give a bijective proof of (\ref{ShiftedMacMahonac}) that
does not involve symmetric functions.

The paper is organized as follows. Section 2 consists of two
subsections. In Subsection \ref{s2.1} we prove Theorem A. In
Subsection \ref{s2.2} we prove Theorem B by showing that
$F_\pi(0,t)=A_\pi(t)$. In Section \ref{MM} we give a bijective proof
of (\ref{ShiftedMacMahonac}).

\begin{acknowledgement}
This work is a part of my doctoral dissertation at California
Institute of Technology and I thank my advisor Alexei Borodin for
all his help.
\end{acknowledgement}
\section{Generalized MacMahon's formula} \lb{s2}
\subsection{Macdonald's case} \lb{s2.1}
We recall a definition of a plane partition. For basics, such as
ordinary partitions and Young diagrams see Chapter 1 of \cite{Mac}.

A plane partition $\pi$ can be viewed in different ways. One way is
to fix a Young diagram, the support of the plane partition, and then
to associate a positive integer to each box in the diagram such that
integers form nonincreasing rows and columns. Thus, a plane
partition is a diagram with row and column nonincreasing integers.
It can also be viewed as a finite two-sided sequence of ordinary
partitions, since each diagonal in the support diagram represents a
partition. We write $ \pi=(\ldots,
\lambda^{-1},\lambda^{0},\lambda^{1}, \ldots ),$ where the partition
$\lambda^{0}$ corresponds to the main diagonal and $\lambda^{k}$
corresponds to the diagonal that is shifted by $k$, see Figure
\ref{PlanePartition}.
Every such two-sided sequence of
partitions represents a plane partition if and only if
\begin{equation}\lb{condpp}
\begin{array}{c}
\cdots \subset \lambda^{-1} \subset \lambda^{0} \supset \lambda^{1}
\supset \cdots  \text{ and}
\medskip
\\
\text{$[\lambda^{n-1}/\lambda^n]$ is a horizontal strip for every
$n$,}
\end{array}
\end{equation}
where
$$
[\lambda/\mu]=
\begin{cases}
\lambda/\mu& \text{if } \lambda \supset \mu,\\
\mu/\lambda& \text{if } \mu \supset \lambda.
\end{cases}
$$
The weight of $\pi$, denoted with $|\pi|$, is the sum of all entries
of $\pi$.

We denote the set of all plane partitions with $\mathcal{P}$ and its
subset containing all plane partitions with at most $r$ nonzero rows
and $c$ nonzero columns with $\mathcal{P}(r,c)$. Similarly, we
denote the set of all ordinary partitions (Young diagrams) with
$\mathcal{Y}$ and those with at most $r$ parts with
$\mathcal{Y}(r)=\mathcal{P}(r,1).$

We use the definitions of $f(n,m)$ and $F_\pi(q,t)$ from the
Introduction. To a plane partition $\pi$ we associate a rational
function $F_\pi(q,t)$ that is related to Macdonald symmetric
functions (for reference see Chapter 6 of \cite{Mac}).


In this section we prove Theorem A.
The proof consists of few steps. We first define weight functions on
sequences of ordinary partitions (Section \ref{s2.1.1}). These
weight functions are defined using Macdonald symmetric functions.
Second, for suitably chosen specializations of these symmetric
functions we obtain that the weight functions vanish for every
sequence of partitions except if the sequence corresponds to a plane
partition (Section \ref{s2.1.2}). Finally, we show that for $\pi \in
\mathcal{P}$ the weight function of $\pi$ is equal to
$F_\pi(q,t)$(Section \ref{s2.1.3}).

Before showing these steps we first comment on a corollary of
Theorem A.

Fix $c=1$. Then, Theorem A gives a generating formula for ordinary
partitions since $\mathcal{P}(r,1)=\mathcal{Y}(r)$. For
$\lambda=(\lambda_1,\lambda_2,\dots)\in \mathcal{Y}(r)$ we define
$d_i=\lambda_i-\lambda_{i+1}$, $i=1,\dots,r$.  Then
$$
F_\lambda(q,t)=\prod_{i=1}^{r}f(d_i,0)=\prod_{i=1}^{r}\prod_{j=1}^{d_i}\frac{1-tq^{j-1}}{1-q^j}.
$$
Note that $F_{\lambda}(q,t)$ depends only on the set of distinct
parts of $\lambda$.
\begin{corollary}\lb{ObicneParticije}
$$
\sum_{\lambda \in
\mathcal{Y}(r)}F_\lambda(q,t)s^{|\lambda|}=\prod_{i=1}^{r}
\frac{(ts^{i};q)_\infty}{(s^{i};q)_{\infty}}.
$$
In particular,
$$
\sum_{\lambda \in
\mathcal{Y}}F_\lambda(q,t)s^{|\lambda|}=\prod_{i=1}^{\infty}
\frac{(ts^{i};q)_\infty}{(s^{i};q)_{\infty}}.
$$
\end{corollary}
This corollary is easy to show directly.
\begin{proof}
First, we expand $(ts;q)_\infty/(s;q)_\infty$ into the power series in
$s$. Let $a_d(q,t)$ be the coefficient of $s^d$. Observe that
$$
\frac{(ts;q)_\infty}{(s;q)_\infty}:=\sum_{d=0}^{\infty}a_d(q,t)s^d=\frac{1-ts}{1-s}\sum_{d=0}^{\infty}a_d(q,t)s^dq^d.
$$
This implies that
$$
a_d(q,t)=f(d,0).
$$

Every $\lambda=(\lambda_1,\dots,\lambda_r)\in \mathcal{Y}(r)$ is
uniquely determined by $d_i\in \bbN \cup\{0\}$, $i=1,\dots,r$, where
$d_i=\lambda_i-\lambda_{i+1}$. Then $\lambda_i=\sum_{j\geq0}
d_{i+j}$ and $|\lambda|=\sum_{i=1}^{r}id_i$. Therefore,
\begin{eqnarray*}
\prod_{i=1}^{r}
\frac{(ts^{i};q)_\infty}{(s^{i};q)_{\infty}}&=&\prod_{i=1}^{r}\sum_{d_i=0}^{\infty}a_{d_i}(q,t)s^{id_i}\\
&=&\sum_{d_1,\dots,d_r}\left[\prod_{i=1}^{r}a_{d_i}(q,t)\right]\cdot
\left[s^{\sum_{i=1}^rid_i}\right]=\sum_{\lambda \in
\mathcal{Y}(r)}F_{\lambda}(q,t)s^{|\lambda|}.
\end{eqnarray*}
\end{proof}

\subsubsection{\bf{The weight functions}} \lb{s2.1.1}
The weight function is defined as a product of Macdonald symmetric
functions $P$ and $Q$. We follow the notation of Chapter 6 of
\cite{Mac}.

Let $\Lambda$ be the algebra of symmetric functions. A
specialization of $\Lambda$ is an algebra homomorphism $\Lambda \to
\bbC$. If $\rho$ and $\sigma$ are specializations of $\Lambda$ then
we write $P_{\lambda/\mu}(\rho;q,t)$, $Q_{\lambda/\mu}(\rho;q,t)$
and $\Pi(\rho,\sigma;q,t)$ for the images of
$P_{\lambda/\mu}(x;q,t)$, $Q_{\lambda/\mu}(x;q,t)$ and
$\Pi(x,y;q,t)$ under $\rho$, respectively $\rho \otimes \sigma$.
Every map $\rho:(x_1,x_2,\dots)\to (a_1,a_2,\dots)$ where $a_i\in
\bbC$ and only finitely many $a_i$'s
are nonzero defines a specialization. 

Let $\rho=(\rho_0^{+},\rho_1^{-},\rho_1^{+}, \ldots, \rho_T^{-})$ be
a finite sequence of specializations. For two sequences of
partitions $ \lambda=(\lambda^{1},\lambda^{2}, \ldots ,\lambda^{T})$
and $ \mu=(\mu^{1},\mu^{2}, \ldots ,\mu^{T-1})$ we set the weight
function $W(\lambda, \mu;q,t)$ to be
$$
W(\lambda, \mu;q,t)=\prod_{n=1}^{T}Q_{\lambda^{n}/\mu^{n-1}}
(\rho_{n-1}^{+};q,t)P_{\lambda^{n}/\mu^{n}}({\rho_n^{-}};q,t),
$$
where $\mu^0=\mu^T=\emptyset$.
Note that $W(\lambda , \mu;q,t)=0$ unless
$$\emptyset \subset \lambda^{1} \supset \mu^{1} \subset \lambda^{2} \supset \mu^{2} \subset \ldots \supset \mu^{T-1} \subset
\lambda^{T} \supset \emptyset.$$

Recall that ((6.2.5) and (6.4.13) of \cite{Mac})
$$
\Pi(x,y;q,t)=\sum_{\lambda \in
\mathcal{Y}}Q_\lambda(x;q,t)P_\lambda(y;q,t)=\prod_{i,j}\frac{(tx_iy_j;q)_\infty}{(x_iy_j;q)_\infty}.
$$

\begin{proposition}\lb{Z}
The sum of the weights $W(\lambda, \mu;q,t)$ over all sequences of
partitions $ \lambda=(\lambda^{1},\lambda^{2}, \ldots ,\lambda^{T})$
and $ \mu=(\mu^{1},\mu^{2}, \ldots ,\mu^{T-1})$ is equal to
\begin{equation} \lb{FormulaForZ}
Z(\rho;q,t)=\prod_{0 \leq i < j \leq T}\Pi(\rho_i^{+},\rho_j^{-};q,t).
\end{equation}
\end{proposition}

\begin{proof}
We use
\begin{equation*}
\sum_{\lambda \in \mathcal{Y}}Q_{\lambda / \mu}(x)P_{\lambda /
\nu}(y)=\Pi(x,y)\sum_{\tau \in \mathcal{Y}}Q_{\nu / \tau}(x)P_{\mu /
\tau}(y).
\end{equation*}
The proof of this is analogous to the proof of Proposition 5.1 that
appeared in our earlier paper \cite{V}. Also, see Example 26 of I.5
of \cite{Mac}.

We prove (\ref{FormulaForZ}) by induction on $T$. Using the formula
above  we substitute sums over $\lambda^{i}$'s with sums over
$\tau^{i-1}$'s as in the proof of Proposition 2.1 of \cite{BR}. This
gives
$$\prod_{i=0}^{T-1}\Pi({\rho_i^{+}},{\rho_{i+1}^{-}}) \sum_{\mu, \tau}Q_{\mu^{1}} (\rho_0^{+})P_{\mu^{1}/\tau^{1}}
({\rho_2^{-}}) Q_{\mu^{2}/\tau^{1}} (\rho_1^{+}) \ldots
P_{\mu^{T-1}} ({\rho_{T}^{-}}).$$ This is the sum of $W(\mu,\tau)$
with  $ \mu=(\mu^{1}, \ldots ,\mu^{T-1})$ and $ \tau=(\tau^{1},
\ldots ,\tau^{T-2})$. Inductively, we obtain (\ref{FormulaForZ}).
\end{proof}

\subsubsection{\bf{Specializations}} \lb{s2.1.2}

For $\pi=(\dots,\lambda^{-1},\lambda^0,\lambda^{1},\dots) \in
\mathcal{P}$ we define a function $\Phi_{\pi}(q,t)$ by
\begin{equation} \lb{alternation}
\Phi_{\pi}(q,t)=\frac{1}{b_{\lambda^0}(q,t)}\prod_{n=-\infty}^{\infty}\varphi_{[\lambda^{n-1}/\lambda^{n}]}(q,t),
\end{equation}
where $b$ and $\varphi$ are given with (6.6.19) and (6.6.24)(i) on
p.341 of \cite{Mac}. Only finitely many terms in the product are
different than 1 because only finitely many $\lambda^{n}$ are
nonempty partitions.

We show that for a suitably chosen specializations the weight
function vanishes for every sequence of ordinary partitions unless
this sequence represents a plane partition in which case it becomes
(\ref{alternation}). This, together with Proposition \ref{Z},
implies
\begin{proposition}\lb{pomocna1}
$$
\sum_{\pi \in
\mathcal{P}(r,c)}\Phi_\pi(q,t)s^{|\pi|}=\prod_{i=1}^{r}
\prod_{j=1}^{c} \frac{(ts^{i+j-1};q)_\infty}{(s^{i+j-1};q)_\infty}.
$$
\end{proposition}

\begin{proof}
If $\rho$ is a specialization of $\Lambda$ where
$x_1=s,\,x_2=x_3=\ldots=0$ then by (6.7.14) and (6.7.14$^\prime$)
of \cite{Mac}
$$
\begin{array}{lcc}
 Q_{\lambda/\mu}(\rho)=
\begin{cases}
\varphi_{\lambda/\mu} s^{|\lambda|-|\mu|} & \text{$\;\;\;\;\;\;\;\;\;\lambda \supset \mu$,  $\lambda/\mu$ is a horizontal strip},\\
0 & \text{$\;\;\;\;\;\;\;\;\;$otherwise},
\end{cases}\\
P_{\lambda/\mu}(\rho)=
\begin{cases}
\varphi_{\lambda/\mu}b_\mu/b_\lambda s^{|\lambda|-|\mu|} & \text{ $\lambda \supset \mu$, $\lambda/\mu$ is a horizontal strip},\\
0 & \text{ otherwise}.
\end{cases}
\end{array}$$

Let
\begin{equation*}
\begin{array}{llll}
\rho_n^+:x_1=s^{-n-1/2},\,x_2=x_3=\ldots=0 &&&-r \leq n \leq -1,\\
\rho_n^-:x_1=x_2=\ldots=0 &&& -r+1 \leq n \leq -1,\\
\rho_n^-:x_1=s^{n+1/2},\,x_2=x_3=\ldots=0 &&&\;\;\,0 \leq n \leq c-1,\\
\rho_n^+:x_1=x_2=\ldots=0 &&&\;\;\,0 \leq n \leq c-2.
\end{array}
\end{equation*}

Then for any two sequences
$\lambda=(\lambda^{-r+1}\dots,\lambda^{-1},\lambda^0,\lambda^1,\dots\lambda^{c-1})$
and $\mu=(\mu^{-r+1}\dots,\mu^{-1},\mu^0,\mu^{1},\dots\mu^{c-2})$
the weight function is given with
\begin{eqnarray*} W(\lambda, \mu)&=&\prod_{n=-r+1}^{c-1} Q_{\lambda^{n}/\mu^{n-1}}
(\rho_{n-1}^{+})P_{\lambda^{n}/\mu^{n}} ({\rho_n^{-}}),
\end{eqnarray*}
where $\mu^{-r}=\mu^{c-1}=\emptyset.$ Then $W(\lambda,\mu)=0$ unless
$$
\mu^{n}=\begin{cases}
\lambda^n&n<0,\\
\lambda^{n+1}&n\geq0,
\end{cases}
$$
$$
\cdots\lambda^{-1} \subset \lambda^0  \supset
\lambda^1\supset\cdots,
$$
$$
\begin{array}{c}
[\lambda^{n-1} /\lambda ^{n}]\text{ is a horizontal strip for every
} n ,
\end{array}
$$
i.e. $\lambda \in \mathcal{P}$ and in that case
\begin{eqnarray*}
W(\lambda, \mu)&=&\prod_{n=-r+1}^{0}\varphi_{\lambda^{n}/\lambda^{n-1}}(q,t)s^{(-2n+1)(|\lambda^{n}|-|\lambda^{n-1}|)/2}\\
&&\cdot\prod_{n=1}^{c}\frac{b_{\lambda^n}(q,t)}{b_{\lambda^{n-1}}(q,t)}\varphi_{\lambda^{n-1}/\lambda^{n}}(q,t)s^{(2n-1)(|\lambda^{n-1}|-|\lambda^{n}|)/2}\\
&=&\frac{1}{b_{\lambda^0}(q,t)}\prod_{n=-r+1}^{c}\varphi_{[\lambda^{n-1}/\lambda^{n}]}(q,t)s^{|\lambda|}=\Phi_{\lambda}(q,t)s^{|\lambda|}.
\end{eqnarray*}

If $\rho^+$ is $x_1=s,\,x_2=x_3=\ldots=0$ and $\rho^-$ is
$x_1=r,\,x_2=x_3=\ldots=0$ then
$$
\Pi(\rho^+,\rho^-)=\prod_{i,\,j} \left. \frac{(tx_iy_j;q)_\infty}{(x_iy_j;q)_\infty}
\right| _{x=\rho^+,\, y=\rho^-}=\frac{(tsr;q)_\infty}{(sr;q)_\infty}.
$$
Then, by Proposition \ref{Z}, for the given specializations of
$\rho_i^+$'s and $\rho_j^-$'s we have
$$
Z=\prod_{i=-1}^{-r} \prod_{j=0}^{c-1}
\Pi(\rho_i^+,\rho_j^-)=\prod_{i=1}^{r} \prod_{j=1}^{c}
\frac{(ts^{i+j-1};q)_\infty}{(s^{i+j-1};q)_\infty}.
$$
\end{proof}


\subsubsection{\bf{Final step}} \lb{s2.1.3}

We show that $F_\pi(q,t)=\Phi_\pi(q,t)$. Then Proposition
\ref{pomocna1} implies Theorem A.

\begin{proposition}
Let $\pi \in \mathcal{P}$. Then
$$
F_{\pi}(q,t)=\Phi_{\pi}(q,t).
$$
\end{proposition}
\begin{proof}
We show this by induction on the number of boxes in the support of
$\pi$. Denote the last nonzero part in the last row of the support
of $\pi$ by $x$. Let $\lambda$ be a diagonal partition containing it
and let $x$ be its $k$th part. Because of the symmetry with respect
to the transposition we can assume that $\lambda$ is one of diagonal
partitions on the left.

Let $\pi'$ be a plane partition obtained from $\pi$ by removing $x$.
We want to show that $F_{\pi}$ and $F_{\pi'}$ satisfy the same
recurrence relation as $\Phi_{\pi}$ and $\Phi_{\pi'}$. The
verification uses the explicit formulas for $b_\lambda$ and
$\varphi_{\lambda/\mu}$ given by (6.6.19) and (6.6.24)(i) on p.341
of \cite{Mac}.

We divide the problem in several cases depending on the position of
the box containing $x$. Let I, II and III be the cases shown in
Figure \ref{Cases}.
\begin{figure} [htp!]
\centering \includegraphics[height=4cm]{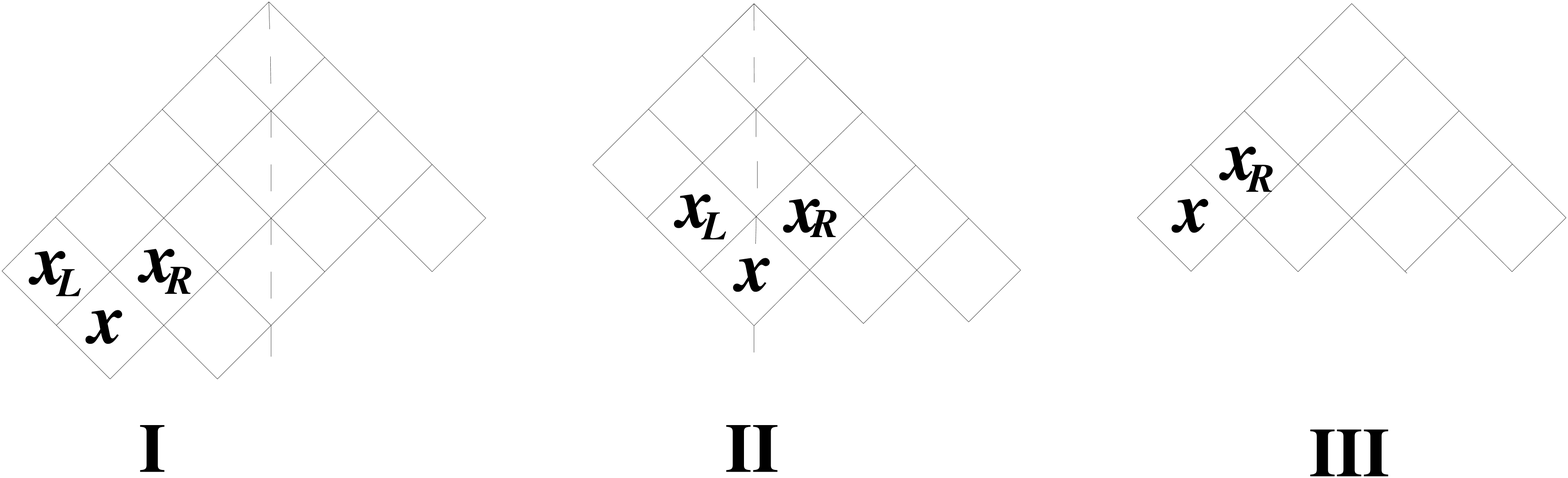}
\caption{Cases I, II and III}
\label{Cases} \end{figure}

Let $\lambda^L$ and $\lambda^R$ be the diagonal partitions of $\pi$
containing $x_L$ and $x_R$, respectively. Let $\lambda'$ be a
partition obtained from $\lambda$ by removing $x$.

If III then $k=1$ and one checks easily that
$$
\frac{F_{\pi'}}{F_{\pi}}=\frac{\Phi_{\pi'}}{\Phi_{\pi}}
=\frac{f(\lambda_1^R,0)}{f(\lambda^R_1-\lambda_1,0)f(\lambda_1,0)}.
$$

Assume I or II. Then
\begin{eqnarray*}
\Phi_{\pi'}=\Phi_{\pi}\cdot
\frac{\varphi_{[\lambda'/\lambda^L]}}{\varphi_{[\lambda/\lambda^L]}}
\cdot\frac{\varphi_{[\lambda'/\lambda^R]}}{\varphi_{[\lambda/\lambda^R]}}
\cdot\frac{b_{\lambda^0(\pi)}}{b_{\lambda^0(\pi')}}=\Phi_{\pi}\cdot \Phi_L
\cdot \Phi_R \cdot \Phi_0.
\end{eqnarray*}
Thus, we need to show that
\begin{equation}\lb{vazi}
\Phi_L \cdot \Phi_R \cdot \Phi_0=F:=\frac{F_{\pi'}}{F_{\pi}}.
\end{equation}

If I then $\lambda^L_{k-1}=x_L$ and $\lambda^R_{k}=x_R$. From the
definition of $\varphi$ we have that
\begin{equation} \label{phiL}
\Phi_L=
{\prod_{i=0}^{k-1}\frac{f(\lambda_{k-i}-\lambda_k,i)}{f(\lambda_{k-i},i)}}
\cdot
{\prod_{i=0}^{k-2}\frac{f(\lambda^L_{k-1-i},i)}{f(\lambda^L_{k-1-i}-\lambda_k,i)}}.
\end{equation}
Similarly,
$$
\Phi_R={\prod_{i=0}^{k-2}\frac{f(\lambda_{k-1-i}-\lambda_k,i)}{f(\lambda_{k-1-i},i)}}
\cdot
{\prod_{i=0}^{k-1}\frac{f(\lambda^R_{k-i},i)}{f(\lambda^R_{k-i}-\lambda_k,i)}}.
$$

If II then $\lambda^L_{k-1}=x_L$ and $\lambda^R_{k-1}=x_R$ and both
$\Phi_L$ and $\Phi_R$ are given with (\ref{phiL}), substituting $L$
with $R$ for $\Phi_R$, while
$$
\Phi_0=\prod_{i=0}^{k-1}\frac{f(\lambda_{k-i},i)}{f(\lambda_{k-i}-\lambda_k,i)}\cdot
{\prod_{i=0}^{k-2}\frac{f(\lambda_{k-1-i}-\lambda_k,i)}{f(\lambda_{k-1-i},i)}}.
$$

From the definition of $F$ one can verify that (\ref{vazi})
holds.
\end{proof}


\subsection{Hall-Littlewood's case} \lb{s2.2}
We analyze the generalized MacMahon's formula in Hall-Littlewood's
case, i.e. when $q=0$, in more detail. Namely, we describe
$F_\pi(0,t)$.

%

We use the definition of $A_\pi(t)$ from the Introduction. In
Proposition \ref{HLcasePol} we show that $F_{\pi}(0,t)=A_{\pi}(t)$.
This, together with Theorem A, implies Theorem B.

Note that the result implies the following simple identities. If
$\lambda\in\mathcal{Y}=\bigcup_{r \geq 1}\mathcal{P}(r,1)$ then
$k(\lambda)$ becomes the number of distinct parts of $\lambda$.
\begin{corollary}
$$
\sum_{\lambda \in
\mathcal{Y}(r)}(1-t)^{k(\lambda)}s^{|\lambda|}=\prod_{i=1}^{r}
\frac{1-ts^i}{1-s^{i}}.
$$
In particular,
$$
\sum_{\lambda \in
\mathcal{Y}}(1-t)^{k(\lambda)}s^{|\lambda|}=\prod_{i=1}^{\infty}
\frac{1-ts^{i}}{1-s^{i}}.
$$
\end{corollary}
These formulas are easily proved by the argument used in the proof
of Corollary \ref{ObicneParticije}.

We now prove
\begin{proposition} \lb{HLcasePol}Let $\pi\in \mathcal{P}$.
Then
$$
F_{\pi}(0,t)=A_{\pi}(t).
$$
\end{proposition}
\begin{proof}
Let $B$ be a $h$-level border component of $\pi$.
Let $F(i,j)=F_\pi(i,j)(0,t)$. It is enough to show that
\begin{equation}\lb{ProdPoBK}
\prod_{(i,j)\in B}F(i,j)=1-t^h.
\end{equation}
Let
\begin{equation*}
c(i,j)=\chi_B(i+1,j)+\chi_B(i,j+1),
\end{equation*}
where $\chi_B$ is the characteristic function of $B$ taking value 1
on the set $B$ and 0 elsewhere. If there are $n$ boxes in $B$ then
\begin{equation}\lb{C}
\sum_{(i,j)\in B}c(i,j)=n-1.
\end{equation}
Let $(i,j)\in B$. We claim that
\begin{equation}\lb{phic}
F(i,j)=(1-t^h)^{1-c(i,j)}.
\end{equation}
Then (\ref{C}) and (\ref{phic}) imply (\ref{ProdPoBK}).

To show (\ref{phic}) we observe that
$$
f(l,m)(0,t)=
\begin{cases}
1&l=0\\
1-t^{m+1}&l\geq1.
\end{cases}
$$
With the same notation as in (\ref{lmn}) we have that $\mu_{m}$,
$\nu_{m}$, $\lambda_{m}$, $\lambda_{m+1}$ are all equal to
$\lambda_1$ for every $m<h$, while for every $m>h$ they are all
different from $\lambda_1$. Then
\begin{eqnarray*}
F(i,j)&=&\prod_{m=0}^{\infty}\frac{f(\lambda_1-\mu_{m+1},m)(0,t)f(\lambda_1-\nu_{m+1},m)(0,t)}
{f(\lambda_1-\lambda_{m+1},m)(0,t)f(\lambda_1-\lambda_{m+2},m)(0,t)}\\
&&=\frac{f(\lambda_1-\mu_{h},h-1)(0,t)f(\lambda_1-\nu_{h},h-1)(0,t)}
{f(\lambda_1-\lambda_{h},h-1)(0,t)f(\lambda_1-\lambda_{h+1},h-1)(0,t)}\\
&&= \frac{(1-t^h)^{1-\chi_{B}(i+1,j)}(1-t^h)^{1-\chi_{B}(i,j+1)}}
{1\cdot (1-t^h)}=(1-t^h)^{1-c(i,j)}.\\
\end{eqnarray*}

\end{proof}
\bigskip

\section{A bijective proof of the shifted MacMahon's formula} \lb{MM}
In this section we are going to give another proof of the shifted
MacMahon's formula (\ref{ShiftedMacMahonac}). More generally, we
prove
\begin{theorem}\lb{genshifMM}
$$
\sum_{\pi \in
\mathcal{SP}(r,c)}2^{k(\pi)}x^{\operatorname{tr}(\pi)}s^{|\pi|}=\prod
_{i=1}^{r}\prod_{j=1}^{c}\frac{1+xs^{i+j-1}}{1-xs^{i+j-1}}.
$$
\end{theorem}
Here $\mathcal{SP}(r,c)$ is the set of strict plane partitions with
at most $r$ rows and $c$ columns. Trace of $\pi$, denoted with
$\operatorname{tr}(\pi)$, is the sum of diagonal entries of $\pi$.

The proof is mostly independent of the rest of the paper. It is
similar in spirit to the proof of MacMahon's formula given in
Section 7.20 of \cite{S}. It uses two bijections. One correspondence
is between strict plane partitions and pairs of shifted tableaux.
The other one is between  pairs of marked shifted tableaux and
marked matrices and it is obtained by the shifted Knuth's algorithm.

We recall the definitions of a marked tableau and a marked shifted
tableau (see e.g. Chapter 13 of \cite{HH}).

Let P be a totally ordered set
$$
P=\{1<1'<2<2'< \cdots\}.
$$
We distinguish elements in $P$ as marked and unmarked, the former
being the one with a prime. We use $|p|$ for the unmarked number
corresponding to $p \in P$.

A marked (shifted) tableau is a (shifted) Young diagram filled with
row and column nonincreasing elements from $P$ such that any given
unmarked element occurs at most once in each column whereas any
marked element occurs at most once in each row. Examples of a marked
tableau and a marked shifted tableau are given in Figure
\ref{Tableaux}.

\begin{figure} [htp!]
\centering \includegraphics[height=3cm]{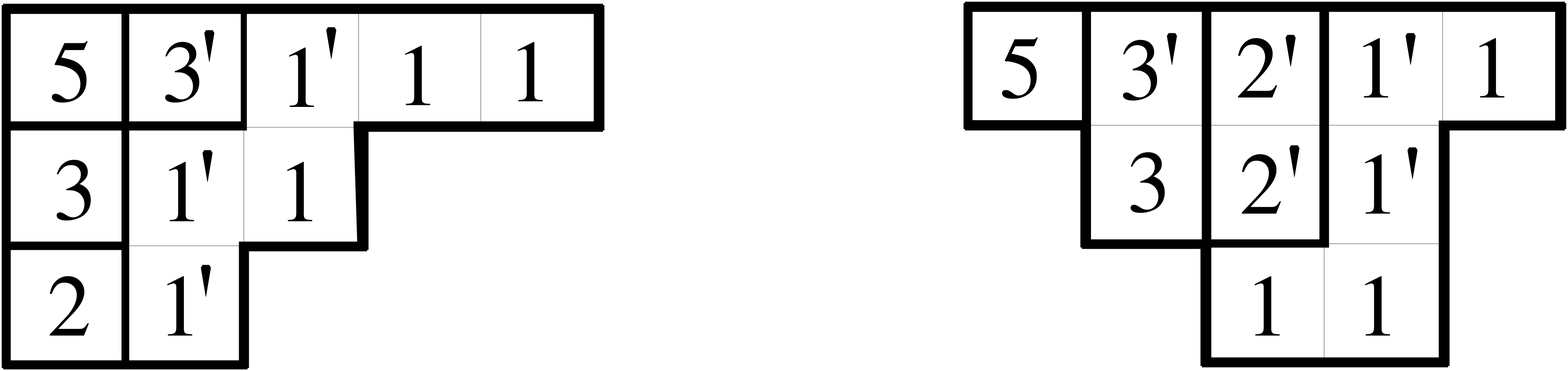}
\caption{ A marked tableau and a marked shifted tableau }
\label{Tableaux} \end{figure}


An unmarked (shifted) tableau is a tableau obtained by deleting
primes from a marked (shifted) tableau. We can also define it as a
(shifted) diagram filled with row and column nonincreasing positive
integers such that no $2 \times 2$ square is filled with the same
number. Unmarked tableaux are strict plane partitions.

We define connected components of a marked or unmarked (shifted)
tableau in a similar way as for plane partitions. Namely, a
connected component is the set of connected boxes filled with $p$ or
$p'$. By the definition of a tableau all connected components are
border strips. Connected components for the examples above are shown
in Figure \ref{Tableaux} (bold lines represent boundaries of these
components).

We use $k(S)$ to denote the number of components of a marked or
unmarked (shifted) tableau $S$. For every marked (shifted) tableau
there is a corresponding unmarked (shifted) tableau obtained by
deleting all the primes. The number of marked (shifted) tableaux
corresponding to the same unmarked (shifted) tableau $S$ is equal to
$2^{k(S)}$ because there are exactly two possible ways to mark each
border component.

For a tableau $S$, we use $\text{sh}(S)$ to denote the shape of $S$
that is an ordinary partition with parts equal to the lengths of
rows of $S$. We define $\ell(S)=\ell(\text{sh}(S))$ and $\max
(S)=|p_{\max}|$, where $p_{\max}$ is the maximal element in $S$. For
both examples $\text{sh}(S)=(5,3,2)$, $\ell(S)=3$ and $\max (S)=5$.

A marked matrix is a matrix with entries from $P \cup \{0\}$.

Let  $\mathcal{ST}^M(r,c)$, respectively $\mathcal{ST}^U(r,c)$, be
the set of ordered pairs $(S,T)$ of marked, respectively unmarked,
shifted tableaux of the same shape where $\max (S)=c$, $\max (T)=r$
and $T$ has no marked letters on its main diagonal. Let
$\mathcal{M}(r,c)$ be the set of $r \times c$ matrices over $P \cup
\{0\}$.


The shifted Knuth's algorithm (see Chapter 13 of \cite{HH}) establishes  the
following correspondence. 
\begin{theorem}\lb{Bij1}
There is a bijective correspondence between matrices $A=[a_{ij}]$
over $P \cup \{0\}$ and ordered pairs $(S,T)$ of marked shifted
tableaux of the same shape such that T has no marked elements on its
main diagonal. The correspondence has the property that $\sum _i
a_{ij}$ is the number of entries $s$ of $S$ for which $|s|=j$ and
$\sum_j a_{ij}$ is the number of entries $t$ of $T$ for which
$|t|=i$.\\
In particular, this correspondence maps  $\mathcal{M}(r,c)$ onto
$\mathcal{ST}^M(r,c)$ and
\begin{eqnarray*}
|\operatorname{sh}(S)|=\sum_{i,j}|a_{ij}|,\;\;\;\;\;
|S|=\sum_{i,j}j|a_{ij}|,\;\;\;\;\; |T|=\sum_{i,j}i|a_{ij}|.
\end{eqnarray*}
\end{theorem}
\begin{remark}
The shifted Knuth's algorithm described in Chapter 13 of \cite{HH}
establishes a correspondence between marked matrices and pairs of
marked shifted tableaux with row and column {\it nondecreasing}
elements. This algorithm can be adjusted to work for marked shifted
tableaux with row and column {\it nonincreasing} elements.  Namely,
one needs to change the encoding of a matrix over $P \cup \{0\}$ and
two algorithms BUMP and EQBUMP, while INSERT, UNMARK, CELL and {\it
unmix} remain unchanged.

One encodes a matrix $A \in \mathcal{P}(r,c)$ into a two-line
notation $E$ with pairs $\begin{array}{c}i\\j\end{array}$ repeated
$|a_{ij}|$ times, where $i$ is going from $r$ to $1$ and $j$ from
$c$ to $1$. If $a_{ij}$ was marked, then we mark the leftmost $j$ in
the pairs $\begin{array}{c}i\\j\end{array}$. The example from p. 246
of \cite{HH}:
$$
A=\left(
\begin{array}{ccc}
1'&0&2\\
2&1&2'\\
1'&1'&0
\end{array}
\right)
$$
would be encoded as
$$
E=\begin{array}{cccccccccc} 3&3&2&2&2&2&2&1&1&1\\
2'&1'&3'&3&2&1&1&3&3&1'
\end{array}.
$$

Algorithms BUMP and EQBUMP insert $x \in P\cup \{0\}$ into a vector
$v$ over $P\cup \{0\}$. By BUMP (resp. EQBUMP) one inserts $x$ into
$v$ by removing (bumping) the leftmost entry of $V$ that is less
(resp. less or equal) than $x$ and replacing it by $x$ or if there
is no such entry then $x$ is placed at the end of $v$.

For the example from above this adjusted shifted Knuth's algorithm
would give
$$
S=
\begin{array}{ccccc}
3'&3&3&3&1'\\
 &2'&2&1&1\\
 & &1'& &\\
\end{array}
\;\;\; \text{and}\;\;\; T=
\begin{array}{ccccc}
3&3&2'&2&2\\
 &2&2&1&1\\
 & &1& &\\
\end{array}
$$
\end{remark}
The other correspondence between pairs of shifted tableaux of the
same shape and strict plane partitions is described in the following
theorem. It is parallel to the correspondence from Section 7.20 of
\cite{S}.
\begin{theorem}\lb{Bij2}
There is a bijective correspondence $\Pi$ between strict plane
partitions $\pi$  and ordered pairs $(S,T)$ of shifted tableaux of
the same shape. This correspondence maps $\mathcal{SP}(r,c)$ onto
$\mathcal{ST}^U(r,c)$ and if $(S,T)=\Pi(\pi)$ then
$$
|\pi|=|S|+|T|-|\operatorname{sh}(S)|,
$$
$$
\operatorname{tr}(\pi)=|\operatorname{sh}(S)|=|\operatorname{sh}(T)|,
$$
$$
k(\pi)=k(S)+k(T)-l(S).
$$
\end{theorem}
\begin{proof}
%

Every $\lambda \in \mathcal{Y}$ is uniquely represented by Frobenius
coordinates $(p_1,\dots,p_d\,|\,q_1,\dots,q_d)$ where $d$ is the
number of diagonal boxes in the Young diagram of $\lambda$ and $p$'s
and $q$'s correspond to the arm length and the leg length, i.e.
$p_i=\lambda_i-i+1$ and $q_i=\lambda'_i-i+1$, where $\lambda'\in
\mathcal{Y}$ is the transpose of $\lambda$.

Let $\pi\in \mathcal{SP}$. Let $(\mu_1,\mu_2,\dots)$ be a sequence
of ordinary partitions whose diagrams are obtained by horizontal
slicing of the 3-dimensional diagram of $\pi$ (see Figure
\ref{3DDiagram}). The Young diagram of $\mu_1$ corresponds to the
first slice and is the same as the support of $\pi$, $\mu_2$
corresponds to the second slice etc. More precisely, the Young
diagram of $\mu_i$ consists of all boxes of the support of $\pi$
filled with numbers greater or equal to $i$. For example, if
$$
\pi=
\begin{array}{ccccc}
5&3&2&1&1\\
4&3&2&1&\\
3&3&2&&\\
2&2&1&&\\
\end{array},
$$
then $(\mu_1,\mu_2,\mu_3,\mu_4,\mu_5)$ are
$$
\mu_1=
\begin{array}{ccccc}
*&*&*&*&*\\
*&*&*&*&\\
*&*&*&&\\
*&*&*&&\\
\end{array},\;\;
\mu_2=
\begin{array}{ccc}
*&*&*\\
*&*&*\\
*&*&*\\
*&*&\\
\end{array},\;\;
\mu_3=
\begin{array}{cc}
*&*\\
*&*\\
*&*\\
\end{array},\;\;
\mu_4=
\begin{array}{c}
*\\
*\\
\end{array},\;\;
\mu_5=
\begin{array}{c}
*\\
\end{array}.
$$
Let $S$, respectively $T$, be an unmarked shifted tableau whose
$i$th diagonal is equal to $p$, respectively $q$, Frobenius
coordinate of $\mu_i$. For the example above
$$
S=
\begin{array}{ccccc}
5&3&2&1&1\\
 &3&2&1&\\
 & &1&1&\\
\end{array}
\;\;\; \text{and}\;\;\; T=
\begin{array}{ccccc}
4&4&3&2&1\\
 &3&3&2&\\
 & &2&1&\\
\end{array}
$$
It is not hard to check that $\Pi$ is a bijection between pairs of
unmarked shifted tableaux of the same shape and strict plane
partitions.

We only verify that
\begin{equation}\lb{okok}
k(\pi)=k(S)+k(T)-l(S).
\end{equation}
Other properties are straightforward implications of the definition
of $\Pi$.
\begin{figure}[htp!]
\begin{center}
\includegraphics[height=7cm]{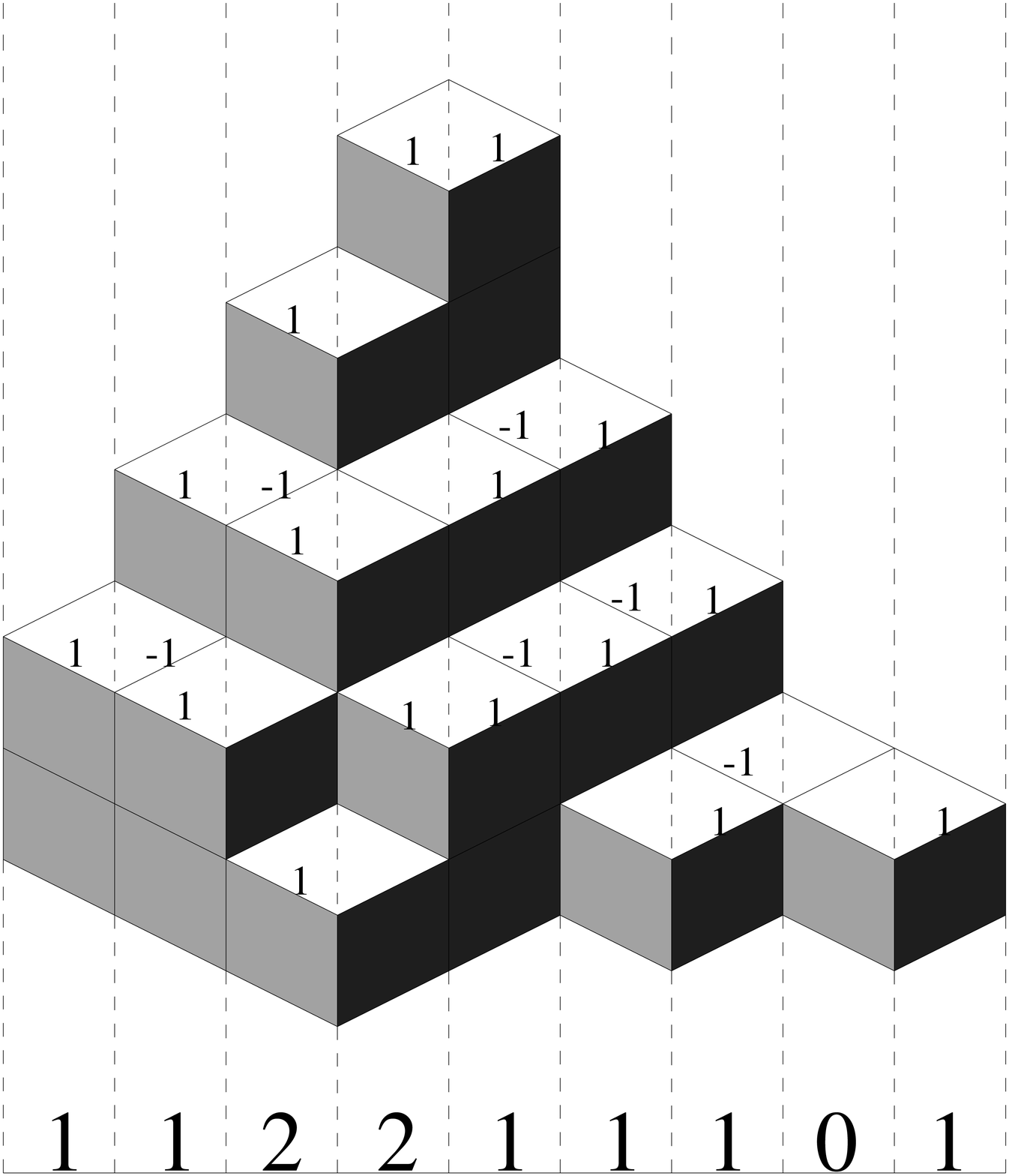}
\end{center}
\caption{3-dimensional diagram of a plane partition}
\label{3DDiagram} \end{figure}

Consider the  3-dimensional diagram of $\pi$ (see Figure
\ref{3DDiagram}) and fix one of its  vertical columns on the right
(with respect to the main diagonal). A rhombus component consists of
all black rhombi that are either directly connected or that have one
white space between them. For the columns on the left we use gray
rhombi instead of black ones. The number at the bottom of each
column in Figure \ref{3DDiagram} is the number of rhombus components
for that column. Let $b$, respectively $g$, be the number of rhombus
components for all right, respectively left, columns.
For the given example $b=4$ and $g=6$.

One can obtain $b$ by a different counting. Consider edges on the right side. Mark all the edges with 0 except the following ones.  Mark a common edge for a white rhombus and a black rhombus where the black rhombus is below the white rhombus with 1. Mark a common edge for two white rhombi that is perpendicular to the plane of black rhombi with -1. See Figure \ref{3DDiagram}. One obtains $b$ by summing these numbers over all edges on the right side of the 3-dimensional diagram.  One recovers $c$ in a similar way by marking edges on the left.

Now, we restrict to a connected component (one of the white terraces, see Figure \ref{3DDiagram}) and sum all the number associated to its edges. If a connected component does not intersect the main diagonal then the sum is equal to 1. Otherwise this sum is equal to 2. This implies that
$$
k(\pi)=b+g-l(\lambda^0).
$$
Since $l(S)=l(\lambda^0)$ it is enough to show that $k(S)=b$ and $k(T)=g$ and (\ref{okok}) follows.

Each black rhombus in the right $i$th column of the 3-dimensional
diagram corresponds to an element of a border strip of $S$ filled
with $i$ and each rhombus component corresponds to a border strip
component. If two adjacent boxes from the same border strip are in
the same row then the corresponding rhombi from the 3-dimensional
diagram are directly connected and if they are in the same column
then there is exactly one white space between them. This implies
$k(S)=b$. Similarly, we get $k(T)=g$.
\end{proof}

Now, using the described correspondences sending $\mathcal{SP}(r,c)$
to $\mathcal{ST}^U(r,c)$ and $\mathcal{ST}^M(r,c)$ to
$\mathcal{M}(r,c)$ we can prove Theorem \ref{genshifMM}.
%
\begin{proof}
\begin{eqnarray*}
\sum_{\pi \in
\mathcal{SP}(r,c)}2^{k(\pi)}x^{\operatorname{tr}(\pi)}s^{|\pi|}&\stackrel{\text{Thm
}\ref{Bij2}}{=}&\sum _{{(S,T) \in
\mathcal{ST}^{{U}}(r,c)}}2^{k(S)+k(T)-l(S)}x^{|\text{sh}S|}s^{|S|+|T|-|\text{sh}S|}\\
&=&\sum
_{{(S,T) \in \mathcal{ST}^{{M}}(r,c)}}x^{|\text{sh}S|}s^{|S|+|T|-|\text{sh}S|}\\
&\stackrel{\text{Thm
}\ref{Bij1}}{=}&\sum_{A \in \mathcal{M}(r,c)}x^{\sum_{i,j}|a_{ij}|}s^{\sum_{i,j} (i+j-1)|a_{ij}|}\\
&=&\prod_{i=1}^r\prod_{j=1}^c \sum_{a_{ij} \in
P\cup{0}}x^{|a_{ij}|}s^{(i+j-1)|a_{ij}|}\\
&=&\prod _{i=1}^{r}\prod_{j=1}^{c}\frac{1+xs^{i+j-1}}{1-xs^{i+j-1}}.
\end{eqnarray*}
\end{proof}
Letting $r \to \infty$ and $c \to \infty$ we get
\begin{corollary}
$$
\sum_{\substack {\pi \in \mathcal{SP}}}
2^{k(\pi)}x^{\text{tr}(\pi)}s^{|\pi|}=\prod_{n=1}^{\infty}\left(\frac{1+xs^{n}}{1-xs^{n}}\right)^n.
$$
\end{corollary}
At $x=1$ we recover the shifted MacMahon's formula.

\bigskip


\end{document}